\documentclass[11pt]{article}

\usepackage[a4paper,margin=28mm]{geometry}
\usepackage[T1]{fontenc}
\usepackage{lmodern}
\usepackage[nopatch=footnote]{microtype}
\usepackage{amsmath,amssymb,amsthm,mathtools}
\usepackage{enumitem}
\usepackage{needspace}
\usepackage{xcolor}
\usepackage[colorlinks=true,linkcolor=blue!55!black,citecolor=blue!55!black,urlcolor=blue!55!black]{hyperref}
\usepackage[nameinlink,noabbrev]{cleveref}
\usepackage{setspace}

\allowdisplaybreaks
\setlist[itemize]{leftmargin=2em,itemsep=0.2em,topsep=0.35em}

\newtheorem{theorem}{Theorem}[section]
\newtheorem{proposition}[theorem]{Proposition}

\theoremstyle{definition}
\newtheorem{definition}[theorem]{Definition}

\newcommand{\R}{\mathbb R}
\newcommand{\cC}{\mathcal C}
\newcommand{\cT}{\mathcal T}
\newcommand{\cR}{\mathcal R}

\newcommand{\dd}{\,\mathrm d}
\newcommand{\dist}{\operatorname{dist}}
\newcommand{\osc}{\operatorname*{ess\,osc}}
\newcommand{\Reg}{\operatorname{Reg}}

\newcommand{\supp}{\operatorname{spt}}
\newcommand{\Lip}{\operatorname{Lip}}
\newcommand{\Tr}{\operatorname{Tr}}
\newcommand{\fintavg}{\mathop{\rlap{\raisebox{0.13ex}{--}}\!\int}\nolimits}

\title{\textbf{The Sharp $C^{1,1}$ Target Threshold for Minimizing Constraint Maps}}
\author{Yilu Liu}
\date{}

\begin{document}
\maketitle

\begin{abstract}
Let $K\subset\R^m$ be the closure of a bounded domain whose boundary is a compact embedded hypersurface, and let
\(
 \overline M=\R^m\setminus\operatorname{int}K
\)
be the allowed target.  We prove that if $\partial K$ is of class
$C^{1,1}$, then every $\overline M$-valued local minimizer of the
Dirichlet energy $E(u;B):=\int_B|Du|^2\dd x$ is locally
$W^{2,\infty}$, and hence locally $C^{1,1}$, on its continuity set.  For
every $0<\alpha<1$, we also construct a convex
body with $C^{1,\alpha}$ boundary and an everywhere continuous global
 minimizer that is not $C^{1,1}$.  Thus $C^{1,1}$ is the sharp target
 regularity threshold in the H\"older scale.
\end{abstract}


\section{Introduction and main results}

Let $\Omega\subset\R^n$ be open, $n\ge1$, and let $K\subset\R^m$, $m\ge2$, be the closure of a bounded domain with compact embedded boundary
\[
 \Gamma:=\partial K.
\]
We regard $K$ as the obstacle in the target and set
\[
 M:=\R^m\setminus K,
 \qquad
 \overline M=\R^m\setminus\operatorname{int}K.
\]
Thus admissible maps are allowed to take values on $\Gamma$ and on the exterior side of $\Gamma$, but not in $\operatorname{int}K$.

\begin{definition}\label{def:minimizer}
A map $u\in W^{1,2}_{\rm loc}(\Omega;\overline M)$ is a \emph{local minimizing constraint map} if, for every ball $B\Subset\Omega$ and every $v\in W^{1,2}(B;\overline M)$ with $v-u\in W^{1,2}_0(B;\R^m)$,
\[
 \int_B |Du|^2\dd x\le \int_B |Dv|^2\dd x.
\]
\end{definition}

The continuity set is defined without choosing a representative:
\[
 \Reg_0(u)
 :=\left\{x\in\Omega:
      \lim_{r\downarrow0}\osc_{B_r(x)}u=0\right\},
 \qquad
 \Sigma(u):=\Omega\setminus\Reg_0(u).
\]
This definition is independent of the representative.  The proof below
shows that $\Reg_0(u)$ is open and supplies a continuous representative
of $u$ there.

The problem combines three strands of regularity theory: harmonic maps,
variational inequalities, and vector-valued maps subject to geometric
constraints.  The variational theory of harmonic maps begins with
Eells--Sampson \cite{EellsSampson1964} and Morrey \cite{Morrey1966}, while
existence and regularity under convexity or curvature assumptions were
developed further by Hildebrandt--Kaul--Widman
\cite{HildebrandtKaulWidman1977}.  The general regularity framework for
vectorial variational integrals is presented in \cite{Giaquinta1983}.

For energy-minimizing harmonic maps, the foundational interior and boundary
partial-regularity results are due to Schoen--Uhlenbeck
\cite{SchoenUhlenbeck1982,SchoenUhlenbeck1983}, with a sharper conclusion
for sphere-valued minimizers in \cite{SchoenUhlenbeck1984}.  Related
constrained models arise in liquid-crystal theory
\cite{HardtKinderlehrerLin1986}, while extensions to $p$-energy and more
general integrands include \cite{HardtLin1987,Luckhaus1988}.  Later work
distinguished the singular behavior of stationary maps
\cite{Bethuel1993,Evans1991,Lin1999} and revealed the role of conservation
laws in critical conformally invariant systems \cite{Riviere2007}.  Broader
accounts of harmonic-map regularity and singularities may be found in
\cite{Helein2002,Simon1996,LinWang2008}.

On the scalar side, the obstacle problem is organized by the
variational-inequality framework of Lions--Stampacchia
\cite{LionsStampacchia1967} and the Lewy--Stampacchia reaction estimate
\cite{LewyStampacchia1969}; see also
\cite{KinderlehrerStampacchia1980}.  Optimal solution regularity for
nonlinear obstacle problems was established in \cite{DuzaarFuchs1986},
whereas the regularity and geometry of the free boundary were developed in
\cite{Caffarelli1998}; systematic treatments appear in
\cite{CaffarelliSalsa2005,PetrosyanShahgholianUraltseva2012}.

For vector-valued maps, a nonconvex target changes both the admissible
variations and the Euler equation; the effects of geometric and topological
constraints already appear in relaxed harmonic-map energies
\cite{BethuelBrezis1991}.  In the target-obstacle setting, Duzaar
\cite{Duzaar1987} obtained a measure-valued Euler equation and $W^{2,p}$
regularity under smooth geometric hypotheses.  The recent work
\cite{FKSobstacle} proved local $C^{1,1}$ regularity for continuous
minimizers when the target boundary is smooth.  In the subsequent survey,
Figalli, Guerra, Kim, and Shahgholian posed Problem~7.1, asking for the
optimal regularity of the target boundary that guarantees local $C^{1,1}$
regularity away from the singular set, and proposed $C^{1,1}$ as the
expected threshold \cite[Problem~7.1]{FGKSreview}.

At this endpoint, the second fundamental form exists only almost
everywhere in the target, so its pullback by $u$ need not be defined on a
full-measure subset of the source.  Our proof instead uses the positive-reach
geometry initiated by Federer \cite{Federer1959} and the Lipschitz regularity
of nearest-point projections on a strictly smaller tube; for a modern
treatment of the latter, see \cite{LeobacherSteinicke2021}.  Curvature enters
only through source-space Lipschitz fields and one-sided quadratic support
inequalities.

The following theorem resolves Problem~7.1 of \cite{FGKSreview} for the
target obstacles considered here.

\begin{theorem}\label{thm:main}
Assume that $\Gamma=\partial K$ is of class $C^{1,1}$ and let $u\in W^{1,2}_{\rm loc}(\Omega;\overline M)$ be a local minimizing constraint map.  Then
\[
 u\in W^{2,\infty}_{\rm loc}(\Reg_0(u);\R^m)
 \subset C^{1,1}_{\rm loc}(\Reg_0(u);\R^m).
\]
In particular, if a representative of $u$ is continuous at
$x_0\in\Omega$, then there exists $r>0$ such that
\[
 u\in W^{2,\infty}(B_r(x_0);\R^m).
\]
\end{theorem}

We next show that the $C^{1,1}$ assumption on the target boundary cannot
be weakened within the H\"older scale, even for convex obstacles.

\begin{theorem}\label{thm:sharpness}
For every $0<\alpha<1$ and every source dimension $n\ge1$, there exist a compact convex body $K_\alpha\subset\R^2$ with
\[
 \partial K_\alpha\in C^{1,\alpha}\setminus C^{1,1},
\]
a bounded Lipschitz domain $\Omega\subset\R^n$, and a map
\[
 u\in C^{1,\alpha}(\overline\Omega;\R^2)
 \cap W^{1,2}(\Omega;\R^2\setminus\operatorname{int}K_\alpha)
\]
which globally minimizes the Dirichlet energy in its fixed Sobolev trace
class but does not belong to $C^{1,1}$ in any neighborhood of a specified
interior point.
\end{theorem}

The counterexample in \cref{thm:sharpness} shows that, for every
$0<\alpha<1$, $C^{1,\alpha}$ regularity of the target boundary alone
does not guarantee local $C^{1,1}$ regularity of minimizing constraint
maps. Moreover, since the H\"older modulus $\omega(r)=r^\alpha$
satisfies
\[
 \int_0^1 \frac{\omega(r)}{r}\,dr
 =\int_0^1 r^{\alpha-1}\,dr
 =\frac{1}{\alpha}<\infty,
\]
the same counterexample shows that $C^1$--Dini regularity alone is also
insufficient.

The proof is organized around four structural propositions.
\Cref{sec:geometry} collects the required $C^{1,1}$ geometry of the target
and constructs an asymptotically nonexpansive retraction.
\Cref{sec:lipschitz} combines harmonic replacement with two Campanato
iterations to upgrade continuity to local Lipschitz regularity.
\Cref{sec:reaction} derives both a bounded scalar reaction and the vector
Euler equation.  \Cref{sec:hessian} combines contact-set identities with
a rotating-normal cancellation estimate for Newton potentials to establish
the endpoint Hessian bound.  All auxiliary estimates are incorporated into
these four arguments, keeping the main proof chain transparent.
\Cref{sec:sharpness} constructs the sharpness examples.

\section{\texorpdfstring{$C^{1,1}$}{C1,1} target geometry}\label{sec:geometry}

Throughout this section, $\Gamma$ is a compact embedded $C^{1,1}$
hypersurface.  Its unit normal $N$ points into the allowed side
$\overline M$.  Define the signed distance by
\begin{equation}\label{eq:signed-distance}
 \rho(y):=
 \begin{cases}
  \dist(y,\Gamma),&y\in\overline M,\\
  -\dist(y,\Gamma),&y\in\operatorname{int}K.
 \end{cases}
\end{equation}
For $\delta>0$ set
\[
 \begin{aligned}
 \cT_\delta
 &:=\{y\in\R^m:\dist(y,\Gamma)<\delta\},\\
 \cT_\delta^+
 &:=\cT_\delta\cap\overline M
   =\{y\in\cT_\delta:\rho(y)\ge0\},\\
 \cT_\delta^-
 &:=\cT_\delta\cap\operatorname{int}K
   =\{y\in\cT_\delta:\rho(y)<0\}.
 \end{aligned}
\]
Thus $\Gamma\subset\cT_\delta^+$, $\cT_\delta^+$ is the allowed part of
the tube, and $\cT_\delta^-$ is the forbidden part.  All constants below
depend only on a fixed finite $C^{1,1}$ atlas of $\Gamma$.

The construction below has a simple purpose.  In a sufficiently small
tubular neighborhood we leave allowed points fixed and project forbidden
points onto $\Gamma$.  The retraction $\cR$ defined in
\eqref{eq:retraction} is locally Lipschitz and satisfies the asymptotic
estimate \eqref{eq:retraction-lip}; in particular, its pointwise Lipschitz
constant tends to one as a forbidden point approaches $\Gamma$.
After a rigid motion, each sufficiently small neighborhood of
$p\in\Gamma$ may be written as
\begin{equation}\label{eq:graph-side}
 \Gamma=\{(z,g(z))\},
 \qquad
 \overline M=\{(z,s):s\ge g(z)\},
\end{equation}
where $g\in C^{1,1}$, with a uniform bound on its $C^{1,1}$ norm.

\begin{proposition}\label{prop:geometry}
There exist $\delta_0>0$ and $C_\Gamma<\infty$ with the following properties.

\begin{enumerate}[label=\textup{(\roman*)}]
\item Every $y\in\cT_{\delta_0}$ has a unique nearest point
$\Pi(y)\in\Gamma$.  The signed distance $\rho$, and the maps $\Pi$ and
$N\circ\Pi$, are Lipschitz on $\cT_{\delta_0}$, and
\begin{equation}\label{eq:normal-coordinates}
 y=\Pi(y)+\rho(y)N(\Pi(y)).
\end{equation}

\item If $p,q\in\Gamma$ and $|p-q|<\delta_0$, then
\begin{equation}\label{eq:quadratic-support}
 N(p)\cdot(q-p)\ge -C_\Gamma|q-p|^2,
 \qquad
 |N(q)-N(p)|\le C_\Gamma|q-p|.
\end{equation}

\item The map
\begin{equation}\label{eq:retraction}
 \cR(y):=
 \begin{cases}
  y,&\rho(y)\ge0,\\
  \Pi(y),&\rho(y)<0,
 \end{cases}
 \qquad y\in\cT_{\delta_0},
\end{equation}
is a locally Lipschitz retraction onto $\overline M\cap\cT_{\delta_0}$ and satisfies
\begin{equation}\label{eq:retraction-lip}
 \operatorname{lip}\cR(y)
 \le 1+C_\Gamma\dist(y,\overline M).
\end{equation}
Consequently, for every $h\in W^{1,2}(U;\cT_{\delta_0})$,
\begin{equation}\label{eq:retraction-chain}
 |D(\cR\circ h)|
 \le \bigl(1+C_\Gamma\dist(h,\overline M)\bigr)|Dh|
 \quad\text{a.e. in }U.
\end{equation}
\end{enumerate}
\end{proposition}

\begin{proof}
\emph{Step 1: quadratic support and tubular coordinates.}
In a graph chart as in \eqref{eq:graph-side}, the Lipschitz continuity of $Dg$ gives
\[
 |g(z)-g(\zeta)-Dg(\zeta)\cdot(z-\zeta)|
 \le \frac{\Lip(Dg)}2|z-\zeta|^2.
\]
Writing the normal at $(\zeta,g(\zeta))$ as
\[
 \frac{(-Dg(\zeta),1)}{\sqrt{1+|Dg(\zeta)|^2}}
\]
yields, after using both orders of the two points,
\begin{equation}\label{eq:two-sided-support}
 |N(p)\cdot(q-p)|\le C_\Gamma|q-p|^2,
 \qquad
 |N(q)-N(p)|\le C_\Gamma|q-p|
\end{equation}
whenever $p,q\in\Gamma$ are sufficiently close.  A finite atlas makes the radius and the constants uniform.  This proves (ii).

By the positive-reach theorem for compact $C^{1,1}$ submanifolds
\cite[Proposition~6]{LeobacherSteinicke2021}, there is $\delta_*>0$
such that every point of $\cT_{\delta_*}$ has a unique nearest point on
$\Gamma$.  If $q=\Pi(y)$, first variation of the squared distance in a
$C^1$ graph chart gives
\[
 y=q+tN(q),
 \qquad
 |t|=\dist(y,\Gamma).
\]
The chosen side fixes the sign of $t$, and uniqueness of the projection
gives
\begin{equation}\label{eq:normal-ray-projection}
 \Pi\bigl(q+tN(q)\bigr)=q
 \qquad (q\in\Gamma,\ |t|<\delta_*).
\end{equation}

Shrink to $0<\delta_0<\delta_*$ so that sufficiently close points of
$\Gamma$ lie in a common graph chart and
\[
 2C_\Gamma\delta_0<\frac12.
\]
For $x,y\in\cT_{\delta_0}$ write
\[
 q:=\Pi(x),
 \qquad
 x=q+sN(q),
 \qquad
 p:=\Pi(y),
 \qquad
 y=p+tN(p),
 \qquad
 |s|,|t|\le r<\delta_0.
\]
If $p$ and $q$ lie in a common chart, \eqref{eq:two-sided-support} gives
\[
 (x-y)\cdot(q-p)
 \ge \bigl(1-2C_\Gamma r\bigr)|q-p|^2,
\]
and consequently
\begin{equation}\label{eq:projection-tube-estimate}
 |\Pi(x)-\Pi(y)|
 \le \frac{1}{1-2C_\Gamma r}|x-y|.
\end{equation}
A Lebesgue number for a finite relatively compact subatlas gives
$\sigma>0$ such that projections outside a common chart satisfy
$|p-q|\ge\sigma$.  After a further decrease of $\delta_0$,
$|x-y|\ge\sigma/2$ for such pairs, and compactness of $\Gamma$ gives the
same uniform Lipschitz bound.  Hence $\Pi$ is Lipschitz on the whole
tube.

By \eqref{eq:signed-distance}, $\rho$ is $1$-Lipschitz when both points
lie on the same side of $\Gamma$.  If $x$ and $y$ lie on opposite sides,
choose $\zeta\in[x,y]\cap\Gamma$.  Then
\[
 |\rho(x)-\rho(y)|
 =\dist(x,\Gamma)+\dist(y,\Gamma)
 \le |x-\zeta|+|y-\zeta|
 =|x-y|.
\]
Thus $\rho$ is $1$-Lipschitz on the whole tube.
Since $N$ and $\Pi$ are Lipschitz, so is $N\circ\Pi$.  Thus
$y\mapsto(\Pi(y),\rho(y))$ and its inverse
\[
 (q,t)\longmapsto q+tN(q)
\]
are Lipschitz by \eqref{eq:normal-ray-projection}, completing (i).

\medskip
\noindent\emph{Step 2: the asymptotically nonexpansive retraction.}
The upper pointwise Lipschitz constant is defined by
\[
 \operatorname{lip}\cR(y)
 :=\limsup_{z\to y,\ z\ne y}
   \frac{|\cR(z)-\cR(y)|}{|z-y|}
\]
Decrease $\delta_0$ further if necessary.  Let
\[
 z\in\cT_{\delta_0}^-,
 \qquad q:=\Pi(z),
 \qquad d:=\dist(z,\Gamma)=-\rho(z)>0.
\]
By \eqref{eq:normal-coordinates}, $z=q-dN(q)$.  Moreover,
$\cR=\Pi$ in a neighborhood of $z$.  If $d<r<\delta_0$, then all points in a sufficiently small neighborhood of $z$ have distance less than $r$ from $\Gamma$.  Applying \eqref{eq:projection-tube-estimate} there and then letting $r\downarrow d$ gives
\begin{equation}\label{eq:retraction-forbidden-lip}
 \operatorname{lip}\cR(z)
 \le \frac{1}{1-2C_\Gamma d}
 \le 1+4C_\Gamma d.
\end{equation}
Enlarge $C_\Gamma$ once and for all to absorb the factor four.

We now estimate $\cR$ between a boundary point and a point on the
forbidden side.  Let $p\in\Gamma$ and let
$z=q-dN(q)\in\cT_{\delta_0}^-$, where
$q=\Pi(z)$ and $d=-\rho(z)>0$.  Then
\eqref{eq:quadratic-support}, with $p$ and $q$ interchanged, gives
\[
 N(q)\cdot(q-p)\le C_\Gamma|q-p|^2.
\]
Consequently,
\[
 |z-p|^2
 =|q-p|^2-2dN(q)\cdot(q-p)+d^2
 \ge(1-2C_\Gamma d)|q-p|^2+d^2.
\]
Since $\cR(z)=q$ and $\cR(p)=p$, this implies
\begin{equation}\label{eq:retraction-interface}
 \limsup_{\substack{z\to p\\ z\in\cT_{\delta_0}^-}}
 \frac{|\cR(z)-\cR(p)|}{|z-p|}
 \le1.
\end{equation}
If $y\in\cT_{\delta_0}^+\setminus\{p\}$, then
$\cR(y)=y$ and $\cR(p)=p$, and therefore
\[
 \frac{|\cR(y)-\cR(p)|}{|y-p|}=1.
\]
Together with \eqref{eq:retraction-interface}, this proves
$\operatorname{lip}\cR(p)\le1$ for every $p\in\Gamma$.

It remains to estimate pairs lying strictly on opposite sides of
$\Gamma$.  Let
\[
 y=(w,s)\in\cT_{\delta_0}^+,
 \qquad
 z\in\cT_{\delta_0}^-,
\]
and assume that $y$ and $q:=\Pi(z)$ lie in the same graph chart.  Write
\[
 q=(\zeta,g(\zeta)),
 \qquad d:=-\rho(z)>0,
 \qquad z=q-dN(q).
\]
Since $y$ lies on the allowed side, $s\ge g(w)$.  The graph
representation and the $C^{1,1}$ Taylor estimate therefore give
\[
 \begin{aligned}
 N(q)\cdot(y-q)
 &\ge
 \frac{g(w)-g(\zeta)-Dg(\zeta)\cdot(w-\zeta)}
      {\sqrt{1+|Dg(\zeta)|^2}}\\
 &\ge-C_\Gamma|w-\zeta|^2
 \ge-C_\Gamma|y-q|^2,
 \end{aligned}
\]
Hence
\[
 |y-z|^2
 =|y-q|^2+2dN(q)\cdot(y-q)+d^2
 \ge(1-2C_\Gamma d)|y-q|^2+d^2.
\]
Thus
\begin{equation}\label{eq:retraction-cross-interface}
 |\cR(y)-\cR(z)|
 =|y-q|
 \le(1-2C_\Gamma d)^{-1/2}|y-z|.
\end{equation}
After another uniform decrease of $\delta_0$ and an enlargement of
$C_\Gamma$, \eqref{eq:retraction-forbidden-lip},
\eqref{eq:retraction-interface}, and
\eqref{eq:retraction-cross-interface}, together with the fact that $\cR$
is the identity on the allowed side, prove local Lipschitz continuity
across the interface and the bound \eqref{eq:retraction-lip}.

For every $0<\delta<\delta_0$, the bounds
\eqref{eq:retraction-forbidden-lip}--\eqref{eq:retraction-cross-interface}
are uniform on $\cT_\delta$.  Hence $\cR$ is uniformly Lipschitz there and
extends as a Lipschitz map to $\overline{\cT_\delta}$.
Now use the absolutely continuous-on-lines representatives of $h$; see
\cite[Theorem~2.1.4]{Ziemer}.  The composition $\cR\circ h$ is absolutely
continuous on almost every coordinate line, and
at almost every point of such a line,
\[
 |\partial_i(\cR\circ h)(x)|
 \le \operatorname{lip}\cR(h(x))|\partial_i h(x)|.
\]
Squaring and summing over $i$ proves \eqref{eq:retraction-chain}, including
on sets mapped into $\{\rho=0\}$.
\end{proof}

\section{From continuity to local Lipschitz regularity}\label{sec:lipschitz}

\begin{proposition}\label{prop:lipschitz}
Let $u$ satisfy the hypotheses of \cref{thm:main}.  If $x_0\in\Reg_0(u)$, then there exists $R>0$ such that
\[
 u\in W^{1,\infty}(B_R(x_0);\R^m).
\]
\end{proposition}

\begin{proof}
\emph{Step 1: quadratic penetration of barycenters.}
Choose the graph chart \eqref{eq:graph-side} over a convex base and work
in a smaller graph cylinder compactly contained in that chart.  Let
$\lambda$ be a probability measure supported on feasible points $(z,s)$
in this smaller cylinder whose diameter is at most $\varepsilon$, and
write
\[
 (\bar z,\bar s)=\int(z,s)\dd\lambda.
\]
Then
\begin{equation}\label{eq:barycentric-penetration}
 \bar s-g(\bar z)\ge-C_\Gamma\varepsilon^2.
\end{equation}
In particular,
\[
 \dist((\bar z,\bar s),\overline M)
 \le C_\Gamma\varepsilon^2.
\]
Because the support of $\lambda$ lies in the allowed graph side
$\overline M=\{(z,s):s\ge g(z)\}$, one has $s\ge g(z)$ there.  The
$C^{1,1}$ Taylor estimate yields
\[
 g(z)\ge g(\bar z)+Dg(\bar z)\cdot(z-\bar z)
       -\frac{\Lip(Dg)}2|z-\bar z|^2.
\]
Integrating and using $\int(z-\bar z)\dd\lambda=0$ proves
\eqref{eq:barycentric-penetration}.  If
$\bar s\ge g(\bar z)$, then $(\bar z,\bar s)\in\overline M$.
If $\bar s<g(\bar z)$, the vertical segment from
$(\bar z,\bar s)$ to $(\bar z,g(\bar z))\in\Gamma$ has length
$g(\bar z)-\bar s$.  Hence in both cases
\[
 \dist((\bar z,\bar s),\overline M)
 \le \bigl(g(\bar z)-\bar s\bigr)_+
 \le C_\Gamma\varepsilon^2,
\]
which proves the asserted distance estimate.

\medskip
\noindent\emph{Step 2: harmonic replacement with quadratic error.}
This comparison is the only use of minimality in the proposition.  There
exist $\varepsilon_*>0$ and $C<\infty$ with the following property.  If
$B_r(a)\Subset\Omega$,
\[
 \operatorname{diam}\operatorname{ess\,ran}(u|_{B_r(a)})
 \le\varepsilon\le\varepsilon_*,
\]
 and this range lies either in a convex ball $Q_{\rm int}\Subset M$ or in a
convex ball $Q\Subset\cT_{\delta_0}$ contained in one target graph chart,
then the harmonic replacement $h$ of $u$ in $B_r(a)$ satisfies
\[
 h-u\in W^{1,2}_0(B_r(a);\R^m),
 \qquad
 \Delta h=0,
\]
and
\begin{equation}\label{eq:harmonic-comparison}
 \int_{B_r(a)}|D(u-h)|^2
 \le C\varepsilon^2\int_{B_r(a)}|Du|^2.
\end{equation}
Let $S:=\overline{\operatorname{ess\,ran}(u|_{B_r(a)})}$.  The Sobolev
trace commutes with Lipschitz compositions; see
\cite[Chapters~13 and~18]{Leoni2017}.  Applied to
$F(y):=\dist(y,S)$, this gives
\[
 \dist(\Tr u,S)=F(\Tr u)=\Tr(F\circ u)=0
 \quad\text{a.e. on }\partial B_r(a).
\]
If $n\ge2$, the Poisson representation is
\[
 h(x)=\int_{\partial B_r(a)}P_{B_r(a)}(x,\xi)\Tr u(\xi)\dd\sigma(\xi),
 \qquad
 P_{B_r(a)}(x,\xi)
 =\frac{r^2-|x-a|^2}
        {|\mathbb S^{n-1}|r|x-\xi|^n},
\]
with
\[
 P_{B_r(a)}(x,\xi)\ge0,
 \qquad
 \int_{\partial B_r(a)}P_{B_r(a)}(x,\xi)\dd\sigma(\xi)=1.
\]
 For $n\ge2$ and each $x\in B_r(a)$, define the probability measure
 \[
  \dd\lambda_x(\xi)
  :=P_{B_r(a)}(x,\xi)\dd\sigma(\xi)
  \quad\text{on }\partial B_r(a).
 \]
 Then
 \[
  h(x)=\int_{\partial B_r(a)}\Tr u(\xi)\dd\lambda_x(\xi),
 \]
 and the push-forward measure $(\Tr u)_\#\lambda_x$ is a probability
 measure supported in $S$.  When $n=1$, writing
 $B_r(a)=(a-r,a+r)$, the harmonic replacement is the affine function
 \[
  h(x)=\frac{a+r-x}{2r}\Tr u(a-r)
       +\frac{x-a+r}{2r}\Tr u(a+r).
 \]
 The two coefficients are nonnegative and sum to one, so this is the same
 probability-average representation with a measure supported on the two
 endpoint traces.  In either dimension, convexity shows that $h$ remains
 in the relevant target ball.  In the boundary case, the tangential
 component of this average remains in the convex base of the chosen graph
 chart.

In the interior case, the probability-average representation and
convexity of $Q_{\rm int}$ give
$h(x)\in Q_{\rm int}\subset M$ for every $x\in B_r(a)$.  Since
$h-u\in W^{1,2}_0(B_r(a);\R^m)$, $h$ is then an admissible competitor.
In the boundary-chart case, applying
\eqref{eq:barycentric-penetration} to $(\Tr u)_\#\lambda_x$ (or to the
corresponding two-point measure when $n=1$) gives
\[
 \dist(h(x),\overline M)\le C\varepsilon^2
 \qquad (x\in B_r(a)).
\]
Thus $\cR\circ h$ is an admissible competitor with the same trace as $u$.  By \eqref{eq:retraction-chain},
\begin{equation}\label{eq:retracted-energy}
 \int_{B_r(a)}|D(\cR\circ h)|^2
 \le (1+C\varepsilon^2)
      \int_{B_r(a)}|Dh|^2.
\end{equation}
Here and below the constant is enlarged to absorb the square, using
$(1+C\varepsilon^2)^2\le1+C'\varepsilon^2$ for
$\varepsilon\le\varepsilon_*$.  Minimality of $u$ and
\eqref{eq:retracted-energy} imply
\begin{equation}\label{eq:minimal-harmonic}
 \int_{B_r(a)}|Du|^2
 \le(1+C\varepsilon^2)
     \int_{B_r(a)}|Dh|^2.
\end{equation}
Since $u-h\in W^{1,2}_0(B_r(a))$ and $h$ is harmonic,
\[
 \int_{B_r(a)}Dh:D(u-h)=0,
\]
so
\begin{equation}\label{eq:pythagoras}
 \int_{B_r(a)}|Du|^2
 =\int_{B_r(a)}|Dh|^2
  +\int_{B_r(a)}|D(u-h)|^2.
\end{equation}
Combining \eqref{eq:minimal-harmonic} and \eqref{eq:pythagoras} proves \eqref{eq:harmonic-comparison}.  All traces are the ordinary Sobolev traces on a ball, so the argument is valid for every radius with $B_r(a)\Subset\Omega$.

\medskip
\noindent\emph{Step 3: the first Campanato iteration.}
Set
\[
 K_r:=\overline{\operatorname{ess\,ran}
                  (u|_{B_r(x_0)})}.
\]
For small $r$, the sets $K_r$ are nonempty nested compact subsets of one
bounded $K_{r_0}$, and
$\operatorname{diam}K_r\to0$ because $x_0\in\Reg_0(u)$.  For completeness,
let $r_j\downarrow0$ and choose $p_j\in K_{r_j}$.  If $k\ge j$, then
\[
 |p_k-p_j|\le\operatorname{diam}K_{r_j}\longrightarrow0.
\]
Hence $(p_j)$ is Cauchy and converges to a point $p_0$.  Closedness and
nesting give $p_0\in K_{r_j}$ for every $j$, and then
$p_0\in K_r$ for every $0<r<r_0$ by choosing $r_j<r$.  The vanishing
diameters give uniqueness.  This is the Cantor intersection theorem; see
\cite[Theorem~2.36]{Rudin1976}.  We have therefore proved
\begin{equation}\label{eq:range-collapse}
 p_0\in\bigcap_{0<r<r_0}K_r,
 \qquad
 \operatorname*{ess\,sup}_{B_r(x_0)}|u-p_0|
 \le\operatorname{diam}K_r\longrightarrow0
 \qquad\text{as }r\downarrow0.
\end{equation}
Choose $R_0>0$ and $0<\varepsilon_0\le\varepsilon_*$ so that
$B_{8R_0}(x_0)\Subset\Omega$, the essential range of $u$ on that ball has
diameter at most $\varepsilon_0$, and it lies in a ball compactly contained
in $M$ when $p_0\in M$, or in a convex ball
$Q\Subset\cT_{\delta_0}$ contained in one graph chart when
$p_0\in\Gamma$.  This choice is possible by \eqref{eq:range-collapse}.
Every $B_r(a)\subset B_{7R_0}(x_0)$ inherits the same range conditions;
hence \eqref{eq:harmonic-comparison} applies uniformly on all such balls.

For $B_r(a)\subset B_{7R_0}(x_0)$, set
\[
 E(a,r):=\int_{B_r(a)}|Du|^2.
\]
Let $h$ be the harmonic replacement on $B_r(a)$.  Since $|Dh|^2$ is subharmonic,
\begin{equation}\label{eq:harmonic-decay}
 \int_{B_{\theta r}(a)}|Dh|^2
 \le \theta^n\int_{B_r(a)}|Dh|^2
 \qquad (0<\theta<1).
\end{equation}
Using $Du=Dh+D(u-h)$, \eqref{eq:harmonic-comparison}, and \eqref{eq:harmonic-decay}, we obtain
\begin{equation}\label{eq:first-decay}
 E(a,\theta r)
 \le \bigl(2\theta^n+C\varepsilon_0^2\bigr)E(a,r).
\end{equation}
Choose $\theta\in(0,1/8)$ and then $\varepsilon_0>0$ so small that
\[
 2\theta^n+C\varepsilon_0^2\le\theta^{n-1}.
\]
Iteration of \eqref{eq:first-decay}, followed by interpolation between consecutive $\theta$-adic radii, gives
\begin{equation}\label{eq:first-morrey-growth}
 E(a,r)\le C_0 r^{n-1}.
\end{equation}
Here $a\in B_{5R_0}(x_0)$ and $0<r<R_0$, after reducing the range of
centers and radii.  The constant $C_0$ depends on $R_0$ and the energy on
$B_{7R_0}(x_0)$ but not on $a$ or $r$.

Poincar\'e's inequality and \eqref{eq:first-morrey-growth} give, uniformly for the balls under consideration,
\begin{equation}\label{eq:first-mean-oscillation}
 \fintavg_{B_r(a)}|u-u_{B_r(a)}|^2
 \le Cr^2\fintavg_{B_r(a)}|Du|^2
 \le Cr.
\end{equation}
By \eqref{eq:first-mean-oscillation},
\[
 \begin{aligned}
 |u_{B_r(a)}-u_{B_{r/2}(a)}|
 &\le \fintavg_{B_{r/2}(a)}|u-u_{B_r(a)}|\\
 &\le C\left(\fintavg_{B_r(a)}
                  |u-u_{B_r(a)}|^2\right)^{1/2}
 \le Cr^{1/2}.
 \end{aligned}
\]
Applying this estimate at the radii $2^{-j}r$ gives
\[
 \sum_{j=0}^{\infty}
 \bigl|u_{B_{2^{-j}r}(a)}-u_{B_{2^{-j-1}r}(a)}\bigr|
 \le Cr^{1/2}\sum_{j=0}^{\infty}2^{-j/2}<\infty.
\]
Thus the dyadic averages form a Cauchy sequence; denote their limit by
$u^*(a)$.  It agrees with $u$ at every Lebesgue point, and the same
telescoping estimate gives
\[
 |u^*(a)-u_{B_r(a)}|
 \le\sum_{j=0}^{\infty}
   |u_{B_{2^{-j}r}(a)}-u_{B_{2^{-j-1}r}(a)}|
 \le Cr^{1/2}.
\]
For $x,y\in B_{4R_0}(x_0)$ set $r:=2|x-y|$ and, when $r$ is small,
let $B^*:=B_{3r}(x)$.  Both $B_r(x)$ and $B_r(y)$ lie in $B^*$, so the
same mean-oscillation estimate gives
\[
 |u_{B_r(x)}-u_{B^*}|+|u_{B_r(y)}-u_{B^*}|
 \le Cr^{1/2}.
\]
Combining the last two displays yields
\begin{equation}\label{eq:half-holder}
 |u^*(x)-u^*(y)|\le L|x-y|^{1/2}
 \qquad (x,y\in B_{4R_0}(x_0)).
\end{equation}
After increasing $L$, the same estimate holds for all
$x,y\in B_{4R_0}(x_0)$, because $u$ is essentially bounded on
$B_{8R_0}(x_0)$.  Thus $u^*$ is a $C^{0,1/2}$ representative of $u$ on
$B_{4R_0}(x_0)$, and
\[
 B_{4R_0}(x_0)\subset\Reg_0(u).
\]
We henceforth write this representative again as $u$.  In particular,
$\Reg_0(u)$ is open.

\medskip
\noindent\emph{Step 4: the endpoint Campanato iteration.}
We now repeat the harmonic comparison using the improved oscillation.
If $L>0$, choose
\[
 0<r_*<\frac{R_0}{4}
 \quad\text{so that}\quad
 2Lr_*^{1/2}\le\varepsilon_*,
 \qquad
 Lr_*^{1/2}\le1.
\]
If $L=0$, the estimate \eqref{eq:half-holder} makes $u$ constant and the
desired conclusion is immediate.  Henceforth assume $L>0$.  For every
$B_r(a)\subset B_{4R_0}(x_0)$ with $0<r\le r_*$,
\[
 \osc_{B_r(a)}u\le 2Lr^{1/2},
\]
and hence \eqref{eq:harmonic-comparison} yields
\begin{equation}\label{eq:second-comparison}
 \int_{B_r(a)}|D(u-h)|^2
 \le C L^2 r E(a,r).
\end{equation}
Define the mean-energy quantity
\[
 \Phi(a,r):=r^{-n}E(a,r).
\]
For $0<\delta\le1$, the weighted inequality
\[
 |A+B|^2\le(1+\delta)|A|^2+(1+\delta^{-1})|B|^2
\]
together with \eqref{eq:harmonic-decay} and \eqref{eq:second-comparison} gives
\begin{align}
 \Phi(a,\theta r)
 &\le\Bigl[1+\delta
       +C\theta^{-n}(1+\delta^{-1})L^2r\Bigr]\Phi(a,r).
\label{eq:phi-pre}
\end{align}
Take $\delta=Lr^{1/2}$ in \eqref{eq:phi-pre}.  Then
\begin{equation}\label{eq:phi-iteration}
 \Phi(a,\theta r)
 \le\bigl(1+C_1Lr^{1/2}\bigr)\Phi(a,r).
\end{equation}
Along a geometric sequence $r_j=\theta^j r_*$, the errors are summable:
\[
 \sum_{j=0}^\infty Lr_j^{1/2}
 =\frac{Lr_*^{1/2}}{1-\theta^{1/2}}<\infty.
\]
Thus the product of the factors in \eqref{eq:phi-iteration} is uniformly bounded.  Starting at a fixed $r_*>0$ and interpolating once more between consecutive scales, we obtain
\begin{equation}\label{eq:mean-energy-bound}
 \sup_{a\in B_{2R_0}(x_0)}
 \sup_{0<r<r_*}
 \fintavg_{B_r(a)}|Du|^2<\infty.
\end{equation}
By the Lebesgue differentiation theorem, \eqref{eq:mean-energy-bound} implies
\[
 Du\in L^\infty(B_{2R_0}(x_0)).
\]
 Renaming the radius proves the proposition.  The argument also covers
 $n=1$.  In that case $B_r(a)=(a-r,a+r)$, and the one-dimensional
 Poincar\'e inequality gives
 \[
  \fintavg_{B_r(a)}|u-u_{B_r(a)}|^2
  \le Cr^2\fintavg_{B_r(a)}|u'|^2.
 \]
 Since \eqref{eq:first-morrey-growth} becomes
 $\int_{B_r(a)}|u'|^2\le C_0$, this yields
 \eqref{eq:first-mean-oscillation}.  The dyadic-average argument then
 gives \eqref{eq:half-holder}; substituting this estimate into
 \eqref{eq:harmonic-comparison} gives \eqref{eq:second-comparison}.
 Iterating \eqref{eq:phi-iteration} and using the summability in the
 display preceding \eqref{eq:mean-energy-bound} gives
 \eqref{eq:mean-energy-bound}.  The Lebesgue differentiation theorem
 therefore yields $u'\in L^\infty(B_{2R_0}(x_0))$.
\end{proof}

\section{The bounded reaction and the vector Euler equation}\label{sec:reaction}

Fix a ball $B\Subset\Omega$ on which \cref{prop:lipschitz} applies.  After
shrinking $B$, its image is either separated from $\Gamma$, in which case
$u$ is harmonic, or lies in $\cT_{\delta_0/4}$ and its projection is
contained in one graph chart.  We work in the latter case and write
\begin{equation}\label{eq:qtn}
 q:=\Pi(u),
 \qquad
 t:=\rho(u)\ge0,
 \qquad
 N:=N\circ q,
 \qquad
 u=q+tN.
\end{equation}
Here $q$ is the nearest boundary point and $t$ is the signed distance to
$\Gamma$; these two coordinates separate tangential motion from the
scalar obstacle constraint.  The maps $q,t,N$ are Lipschitz in the source.
Since $q$ takes values in $\Gamma$ and $|N|=1$,
\begin{equation}\label{eq:orthogonality}
 N\cdot Dq=0,
 \qquad
 N\cdot DN=0
 \quad\text{a.e.},
\end{equation}
and
\begin{equation}\label{eq:Du-decomp}
 Du=Dq+tDN+N\otimes Dt,
 \qquad
 |Du|^2=|Dq+tDN|^2+|Dt|^2.
\end{equation}

\begin{proposition}\label{prop:reaction}
Set
\[
 b:=Dq:DN,
 \qquad
 c:=|DN|^2.
\]
There exists $\eta\in L^\infty(B)$ such that
\begin{equation}\label{eq:eta}
 \eta=-\Delta t+b+ct
 \quad\text{in }\mathcal D'(B),
 \qquad
 0\le\eta\le\|b_+\|_{L^\infty(B)}
 \quad\text{a.e.},
\end{equation}
and
\begin{equation}\label{eq:eta-support}
 \eta=0\quad\text{a.e. on }\{t>0\},
 \qquad
 \supp\eta\subset\cC:=\{t=0\}.
\end{equation}
Moreover,
\begin{equation}\label{eq:vector-equation}
 -\Delta u=N\eta
 \quad\text{in }\mathcal D'(B;\R^m),
\end{equation}
and, for every $1<p<\infty$,
\begin{equation}\label{eq:W2p-u}
 t\in W^{2,p}_{\rm loc}(B),
 \qquad
 u\in W^{2,p}_{\rm loc}(B;\R^m).
\end{equation}
\end{proposition}

\begin{proof}
\emph{Step 1: the scalar normal variational inequality.}
The functions $b,c$ belong to $L^\infty(B)$ and $c\ge0$.  For every $\sigma\in W^{1,2}(B)$ with $\sigma\ge0$ and $\sigma-t\in W^{1,2}_0(B)$, we claim that
\begin{equation}\label{eq:scalar-VI}
 \int_B Dt\cdot D(\sigma-t)
 +\int_B(b+ct)(\sigma-t)\ge0.
\end{equation}
It is enough first to consider bounded $\sigma$.  For $0<s<1$, put
\[
 t_s:=(1-s)t+s\sigma,
 \qquad
 u_s:=q+t_sN.
\]
Because $t_s\ge0$, because $t$ stays in the smaller tube, and because $\sigma$ is bounded, there exists $s_0=s_0(\|t\|_{L^\infty},\|\sigma\|_{L^\infty},\delta_0)>0$ such that $u_s$ remains in the allowed part of the tubular neighborhood whenever $0<s<s_0$.  Moreover, $u_s-u\in W^{1,2}_0(B;\R^m)$.  Minimality and \eqref{eq:Du-decomp} give
\[
 0\le \frac{1}{2s}\int_B\bigl(|Du_s|^2-|Du|^2\bigr).
\]
Since $t_s=t+s(\sigma-t)$, passage to the limit gives
\[
 \begin{aligned}
 0
 &\le\int_B Dt\cdot D(\sigma-t)
   +\int_B(Dq+tDN):DN\,(\sigma-t)\\
 &=\int_B Dt\cdot D(\sigma-t)
   +\int_B(b+ct)(\sigma-t),
 \end{aligned}
\]
which is \eqref{eq:scalar-VI}.  For general nonnegative $\sigma$, apply
the bounded case to $\sigma_k:=\min\{\sigma,k\}$ with
$k>\|t\|_{L^\infty(B)}$ and let $k\to\infty$ in $W^{1,2}(B)$.

\medskip
\noindent\emph{Step 2: the bounded Lewy--Stampacchia reaction.}
This is the local Lewy--Stampacchia estimate; compare
\cite{LewyStampacchia1969}.  We include the penalization argument because
the proof of \cref{prop:hessian} uses the exact $L^\infty$ bound for the
reaction $\eta$ obtained below.  On a ball
$B'\Subset B$, keep the trace of $t$ fixed and minimize
\[
 J_\varepsilon(s)
 :=\int_{B'}\left(\frac12|Ds|^2+bs+\frac12cs^2
                  +\frac{1}{2\varepsilon}(s_-)^2\right)\dd x,
 \qquad s-t\in W^{1,2}_0(B'),
\]
where $s_-:=\max\{-s,0\}$.  The unique minimizer $s_\varepsilon$ satisfies
\begin{equation}\label{eq:penalized}
 -\Delta s_\varepsilon+cs_\varepsilon+b
 =\eta_\varepsilon,
 \qquad
 \eta_\varepsilon:=\frac{(s_\varepsilon)_-}{\varepsilon}\ge0.
\end{equation}
Let $M_b:=\|b_+\|_{L^\infty(B')}$.  On the set where
$s_\varepsilon+\varepsilon M_b<0$, one has $s_\varepsilon<0$, and
\eqref{eq:penalized} gives
\[
 -\Delta(s_\varepsilon+\varepsilon M_b)
 +(c+\varepsilon^{-1})(s_\varepsilon+\varepsilon M_b)
 =-b+c\varepsilon M_b+M_b\ge0.
\]
 The trace of $s_\varepsilon+\varepsilon M_b$ is
$t+\varepsilon M_b\ge0$.  Hence
$(s_\varepsilon+\varepsilon M_b)_-\in W^{1,2}_0(B')$, and testing on
its support gives
 \[
  -\int_{B'}|D(s_\varepsilon+\varepsilon M_b)_-|^2
  -\int_{B'}(c+\varepsilon^{-1})
        (s_\varepsilon+\varepsilon M_b)_-^2
  =\int_{B'}(-b+c\varepsilon M_b+M_b)
        (s_\varepsilon+\varepsilon M_b)_-\ge0.
 \]
 The left-hand side is nonpositive; hence both sides vanish and
$s_\varepsilon+\varepsilon M_b\ge0$.  Consequently
\[
 0\le\eta_\varepsilon\le M_b.
\]
 We now prove the required convergence.  Set
 \[
  F(s):=\int_{B'}\left(\frac12|Ds|^2+bs+\frac12cs^2\right)\dd x
 \]
 on the affine space $t+W^{1,2}_0(B')$.  Set
 \[
  \mathcal K_{B'}
  :=\{s\in t+W^{1,2}_0(B'):s\ge0\ \text{a.e.}\}.
 \]
 If $s\in\mathcal K_{B'}$, extend $s-t\in W^{1,2}_0(B')$ by zero to
 $B$ and denote the extension by $\widetilde{s-t}$.  Then
 \[
  \widetilde s:=t+\widetilde{s-t}\in W^{1,2}(B),
  \qquad
  \widetilde s\ge0,
  \qquad
  \widetilde s-t\in W^{1,2}_0(B).
 \]
 Applying \eqref{eq:scalar-VI} to $\widetilde s$ gives the explicit
 restricted inequality
 \begin{equation}\label{eq:restricted-VI}
  \int_{B'}Dt\cdot D(s-t)
  +\int_{B'}(b+ct)(s-t)\ge0.
 \end{equation}
 Consequently, for every $s\in\mathcal K_{B'}$,
 \[
  \begin{aligned}
  F(s)-F(t)
  &=\int_{B'}Dt\cdot D(s-t)
    +\int_{B'}(b+ct)(s-t)\\
  &\quad+\frac12\int_{B'}
       \bigl(|D(s-t)|^2+c|s-t|^2\bigr)\ge0.
  \end{aligned}
 \]
 Thus $t$ is the unique minimizer of $F$ over $\mathcal K_{B'}$;
 uniqueness follows from the last display and Poincar\'e's inequality.
 Since $t\ge0$, minimality of
$s_\varepsilon$ gives
 \begin{equation}\label{eq:penalty-energy}
  F(s_\varepsilon)
  +\frac{1}{2\varepsilon}\|(s_\varepsilon)_-\|_{L^2(B')}^2
  \le F(t).
 \end{equation}
 Coercivity on $t+W^{1,2}_0(B')$, which follows from Poincar\'e's inequality and $c\ge0$, shows that $(s_\varepsilon)$ is bounded in $W^{1,2}(B')$.  The same inequality then gives
 \begin{equation}\label{eq:negative-part-bound}
  \|(s_\varepsilon)_-\|_{L^2(B')}^2\le C\varepsilon.
 \end{equation}
 After extraction, $s_\varepsilon\rightharpoonup s$ in $W^{1,2}(B')$ and
 $s_\varepsilon\to s$ in $L^2(B')$.  Since the map
 $r\mapsto r_-$ is $1$-Lipschitz, \eqref{eq:negative-part-bound} gives
 \[
  \|s_-\|_{L^2(B')}
  \le \|s_--(s_\varepsilon)_-\|_{L^2(B')}
      +\|(s_\varepsilon)_-\|_{L^2(B')}
  \le \|s-s_\varepsilon\|_{L^2(B')}+C\varepsilon^{1/2}
  \longrightarrow0.
 \]
 Thus $s\ge0$ almost everywhere.  Weak lower semicontinuity together
 with \eqref{eq:penalty-energy} gives
 \[
  F(s)\le\liminf_{\varepsilon\downarrow0}F(s_\varepsilon)
  \le F(t).
 \]
 By uniqueness of the constrained minimizer, $s=t$.  The same inequalities imply $F(s_\varepsilon)\to F(t)$.  Since $s_\varepsilon\to t$ in $L^2$, the terms containing $b$ and $c$ converge; hence
 \[
  \|Ds_\varepsilon\|_{L^2(B')}
  \longrightarrow\|Dt\|_{L^2(B')}.
 \]
 Weak convergence of the gradients and convergence of their norms yield
 \[
  s_\varepsilon\longrightarrow t
  \quad\text{strongly in }W^{1,2}(B').
 \]
 Thus the whole family converges, because every weakly convergent subsequence has the same limit.  The bound $0\le\eta_\varepsilon\le M_b$ permits a further weak-* extraction in $L^\infty(B')$.  Passing to the limit in \eqref{eq:penalized} defines
 \[
  \eta=-\Delta t+b+ct
 \]
 and proves \eqref{eq:eta} on $B'$.  The distribution on the right uniquely
determines $\eta$, so the local weak-* limits agree on overlaps.  Since
$B'\Subset B$ was arbitrary, an exhaustion yields \eqref{eq:eta} on $B$.

Let $U\Subset\{t>0\}$ and $\phi\in C_c^\infty(U)$.  Since $t$ is
continuous and positive on $\supp\phi$, both $t+s\phi$ and $t-s\phi$
are admissible in \eqref{eq:scalar-VI} for all sufficiently small $s>0$.
Substituting these two functions into \eqref{eq:scalar-VI} gives,
respectively,
\[
 s\langle-\Delta t+b+ct,\phi\rangle\ge0,
 \qquad
 -s\langle-\Delta t+b+ct,\phi\rangle\ge0.
\]
Since $s>0$, it follows that
\[
 \langle-\Delta t+b+ct,\phi\rangle=0.
\]
Hence $\eta=0$ almost everywhere on $\{t>0\}$ and
$\supp\eta\subset\{t=0\}$, proving \eqref{eq:eta-support}.

The equation
\[
 -\Delta t=\eta-b-ct\in L^\infty(B)
\]
gives $t\in W^{2,\infty}_{\rm loc}(B)$ when $n=1$.  If $n\ge2$, choose
$B''\Subset B'\Subset B$ and $\chi\in C_c^\infty(B')$ with
$\chi\equiv1$ on $B''$, and let
\[
 w:=\Gamma_n*\bigl(\chi(\eta-b-ct)\bigr).
\]
By \cite[Theorem~9.9]{GilbargTrudinger},
$w\in W^{2,p}_{\rm loc}(\R^n)$ for every $1<p<\infty$, while $t-w$ is
distributionally harmonic on $B''$.  Weyl's lemma gives
$t\in W^{2,p}(B'')$ for every finite $p$.

\medskip
\noindent\emph{Step 3: identification of the vector reaction.}
We identify the vector equation at the level of Radon measures, avoiding
any differentiation of $Dg(v)$ with respect to the variation parameter.
By the choice of $B$, after a rigid motion in the target its image lies
in a graph cylinder
\[
 u=(v,h),
 \qquad
 \overline M=\{(z,s):s\ge g(z)\},
 \qquad
 g\in C^{1,1}.
\]
Let $\phi\in C_c^\infty(B)$ be nonnegative and, for sufficiently small
$\varepsilon>0$, set
\[
 u_\varepsilon:=(v,h+\varepsilon\phi).
\]
Because $h+\varepsilon\phi\ge h\ge g(v)$ and the variation has compact
support, $u_\varepsilon$ is admissible.  Minimality gives
\[
 0\le \int_B\bigl(|Du_\varepsilon|^2-|Du|^2\bigr)
 =2\varepsilon\int_B Dh\cdot D\phi
  +\varepsilon^2\int_B|D\phi|^2.
\]
Dividing by $2\varepsilon$ and letting $\varepsilon\downarrow0$ yields
$\int_BDh\cdot D\phi\ge0$.  Hence the distribution
\[
 \mu:=-\Delta h
\]
is nonnegative.  To see directly that it is a Radon measure, fix a compact
set $K_0\Subset B$ and choose $\zeta\in C_c^\infty(B)$ with
$\zeta\ge0$ everywhere and $\zeta\ge1$ on $K_0$.  If
$\phi\in C_c^\infty(B)$ and $\supp\phi\subset K_0$, then
\[
 -\|\phi\|_{L^\infty}\zeta
 \le \phi\le
 \|\phi\|_{L^\infty}\zeta.
\]
Positivity therefore gives
\[
 |\langle\mu,\phi\rangle|
 \le \|\phi\|_{L^\infty}\langle\mu,\zeta\rangle
 \qquad\bigl(\supp\phi\subset K_0\bigr).
\]
Thus $\mu$ has order zero on compact subsets, and the Riesz representation
theorem identifies it with a nonnegative Radon measure.

Set $U:=\{h>g(v)\}$.  This set is open because $u$ and $g$ are
continuous.  For $\phi\in C_c^\infty(U)$, compactness of $\supp\phi$
gives
\[
 \delta:=\min_{\supp\phi}\bigl(h-g(v)\bigr)>0.
\]
If $0<\varepsilon<\delta/\|\phi\|_{L^\infty}$, with the assertion
being trivial when $\phi=0$, then both
$(v,h+\varepsilon\phi)$ and $(v,h-\varepsilon\phi)$ are admissible.
Applying minimality to these two variations gives
\[
 -\frac{\varepsilon}{2}\int_B|D\phi|^2
 \le \int_BDh\cdot D\phi
 \le \frac{\varepsilon}{2}\int_B|D\phi|^2.
\]
Letting $\varepsilon\downarrow0$ shows that
$\langle\mu,\phi\rangle=0$ for every $\phi\in C_c^\infty(U)$.  Therefore
\begin{equation}\label{eq:mu-support}
 \supp\mu\subset\{h=g(v)\}.
\end{equation}

Fix $\varphi\in C_c^\infty(B;\R^{m-1})$.  Since $v$ is continuous and
$\supp\varphi\Subset B$, the graph chart may be chosen so that
$v(\supp\varphi)$ is compactly contained in its base.
For $|s|$ small set
\begin{equation}\label{eq:tangent-variation}
 \delta_s:=g(v+s\varphi)-g(v),
 \qquad
 u_s:=(v+s\varphi,h+\delta_s).
\end{equation}
Indeed,
\[
 h+\delta_s-g(v+s\varphi)=h-g(v)\ge0,
\]
and $v+s\varphi$ remains in the graph base.  Since $Dg$ is Lipschitz
and $v,\varphi$ are Lipschitz,
\begin{equation}\label{eq:delta-estimates}
 \|\delta_s\|_{W^{1,\infty}(B)}\le C|s|,
 \qquad
 \frac{\delta_s}{s}\longrightarrow Dg(v)\cdot\varphi
 \quad\text{uniformly as }s\to0.
\end{equation}
Minimality gives $E(u_s)-E(u)\ge0$.  Hence the quotient by $2s$ is
nonnegative for $s>0$ and nonpositive for $s<0$.  Its explicit
expansion is
\begin{align*}
 \frac{E(u_s)-E(u)}{2s}
 &=\int_B Dv:D\varphi
   +\frac{s}{2}\int_B|D\varphi|^2
   +\frac1s\int_B Dh\cdot D\delta_s
   +\frac{1}{2s}\int_B|D\delta_s|^2.
\end{align*}
The last term tends to zero by \eqref{eq:delta-estimates}.  Since
$-\Delta h=\mu$ and $\delta_s$ is compactly supported and Lipschitz,
approximate $\delta_s$ uniformly and in $W^{1,2}$ by smooth
compactly supported functions.  This extends the distributional
identity to $\delta_s$ and gives
\[
 \frac1s\int_B Dh\cdot D\delta_s
 =\int_B\frac{\delta_s}{s}\dd\mu
 \longrightarrow\int_B Dg(v)\cdot\varphi\dd\mu.
\]
The right and left quotients converge to the same finite limit; their opposite variational inequalities force that limit to be zero.  Thus
\[
 \int_B Dv:D\varphi
 +\int_B Dg(v)\cdot\varphi\dd\mu=0.
\]
Therefore, as vector-valued Radon measures,
\begin{equation}\label{eq:graph-residual}
 -\Delta u=(-Dg(v),1)\mu.
\end{equation}
By \eqref{eq:mu-support}, the graph unit normal agrees $\mu$-a.e. with the tubular normal $N$.  Hence
\begin{equation}\label{eq:normal-measure}
 -\Delta u=N\sqrt{1+|Dg(v)|^2}\,\mu,
\end{equation}
where the nonnegative measure on the right is supported on
$\cC=\{t=0\}$.

It remains to identify this measure with $\eta$.  For
$\phi\in C_c^\infty(B)$, approximate the Lipschitz field $N\phi$ by
smooth compactly supported fields, uniformly and in $W^{1,2}$.  This
permits its use in both the measure and distributional identities.  By
\eqref{eq:Du-decomp}--\eqref{eq:orthogonality},
\begin{align}
 \langle-\Delta u,N\phi\rangle
 &=\int_B Du:D(N\phi)\dd x\notag\\
 &=\int_B\bigl(Dq:DN+t|DN|^2\bigr)\phi\dd x
   +\int_B Dt\cdot D\phi\dd x\notag\\
 &=\langle-\Delta t+b+ct,\phi\rangle
 =\int_B\eta\phi\dd x.
\label{eq:lambda-eta-pair}
\end{align}
On the other hand, \eqref{eq:normal-measure} gives
\[
 \langle-\Delta u,N\phi\rangle
 =\int_B\phi\sqrt{1+|Dg(v)|^2}\dd\mu.
\]
Thus
\[
 \sqrt{1+|Dg(v)|^2}\,\mu=\eta\,\mathcal L^n,
\]
which proves \eqref{eq:vector-equation}.  Since $N\eta\in L^\infty$,
the one-dimensional equation gives \eqref{eq:W2p-u} when $n=1$.  If
$n\ge2$, fix $B''\Subset B'\Subset B$, choose
$\chi\in C_c^\infty(B')$ with $\chi\equiv1$ on $B''$, and set
\[
 V:=\Gamma_n*(\chi N\eta).
\]
By \cite[Theorem~9.9]{GilbargTrudinger}, $V\in W^{2,p}_{\rm loc}(\R^n)$
for every $1<p<\infty$.  Moreover,
\[
 -\Delta(u-V)=N\eta-\chi N\eta=0
 \quad\text{in }\mathcal D'(B'';\R^m).
\]
Weyl's lemma therefore gives $u-V\in C^\infty(B'';\R^m)$, and hence
$u\in W^{2,p}(B'';\R^m)$ for every finite $p$.  Since $B''\Subset B$
was arbitrary, this proves \eqref{eq:W2p-u}.
\end{proof}

\section{The endpoint Hessian estimate}\label{sec:hessian}

Assume $n\ge2$; if $n=1$, \eqref{eq:vector-equation} immediately gives
$u\in W^{2,\infty}_{\rm loc}$.  Work in nested balls
\[
 B_{4R}\Subset B
\]
small enough that the target tubular coordinates above are valid, that
$q(B_{4R})$ is contained in one quadratic-support neighborhood from
\cref{prop:geometry}, and that
\begin{equation}\label{eq:normal-close-scale}
 |N(z)-N(y)|\le\frac14
 \quad\text{whenever }z,y\in B_{4R}.
\end{equation}
For $y\in B_{4R}$ set
\[
 N_y:=N(y),
 \qquad
 P_y:=I-N_y\otimes N_y.
\]

\begin{samepage}
\begin{proposition}\label{prop:hessian}
Assume $n\ge2$, let $u,q,t,N,$ and $\eta$ be as in
\cref{prop:reaction}, and let $B_{4R}\Subset B$ be chosen so that the
tubular coordinates are valid on $u(B_{4R})$, the set $q(B_{4R})$ is
contained in one quadratic-support neighborhood from
\cref{prop:geometry}, and \eqref{eq:normal-close-scale} holds.  Then there
is a constant $C<\infty$,
depending only on $n$, $R$, the distance from $B_{4R}$ to $\partial B$,
the local bounds for $u$, $Du$, $N$, $DN$, and $\eta$, and the
$C^{1,1}$ geometric constants of $\Gamma$, such that
\[
 \|D^2u\|_{L^\infty(B_R)}\le C.
\]
\end{proposition}
\end{samepage}

\begin{proof}
\emph{Step 1: pointwise Newton-potential cancellation.}
Let $x_0\in\R^n$, $n\ge2$, and let $G\in L^\infty_c(\R^n;\R^k)$ satisfy
\[
 |G(z)|\le L|z-x_0|
 \quad\text{for a.e. }z,
 \qquad
 \supp G\subset B_\varrho(x_0).
\]
If $\Gamma_n$ is the fundamental solution of $-\Delta$ and
\[
 V(x):=\int_{\R^n}\Gamma_n(x-z)G(z)\dd z,
\]
then $V$ is twice differentiable at $x_0$ and
\begin{equation}\label{eq:newton-cancel}
 |D^2V(x_0)|\le C_nL\varrho.
\end{equation}
The argument applies componentwise.  The candidate Hessian is absolutely
convergent because
\[
 \int_{B_\varrho(x_0)}
 |D^2\Gamma_n(x_0-z)|\,|G(z)|\dd z
 \le C_nL\int_0^\varrho r^{-n}r\,r^{n-1}\dd r
 =C_nL\varrho.
\]
For a unit vector $e$ and $\tau\ne0$, set
\[
 Q_\tau(e)
 :=\frac{DV(x_0+\tau e)-DV(x_0)}{\tau}
   -\int_{\R^n}D^2\Gamma_n(x_0-z)e\,G(z)\dd z.
\]
Split $Q_\tau(e)$ into the regions
$B_\tau:=\{|z-x_0|\le2|\tau|\}$ and $B_\tau^c$.  The kernel estimates
$|D\Gamma_n(w)|\le C_n|w|^{1-n}$ and
$|D^2\Gamma_n(w)|\le C_n|w|^{-n}$ give
\[
 \begin{aligned}
 |Q_\tau^{\rm near}(e)|
 &\le \frac{C_n}{|\tau|}
   \int_{B_\tau}
   \bigl(|x_0+\tau e-z|^{1-n}+|x_0-z|^{1-n}\bigr)|G(z)|\dd z\\
 &\quad+C_n\int_{B_\tau}|x_0-z|^{-n}|G(z)|\dd z
 \le C_nL|\tau|.
 \end{aligned}
\]
On $B_\tau^c$, the mean-value theorem yields
\[
 \left|
 \frac{D\Gamma_n(x_0+\tau e-z)-D\Gamma_n(x_0-z)}{\tau}
 \right|
 \le C_n|z-x_0|^{-n}.
\]
After multiplication by $G$, the majorant is
$C_nL|z-x_0|^{1-n}\mathbf1_{B_\varrho(x_0)}$, which is integrable.
The convergence is uniform in $e$: for $0<2|\tau|<\delta$, the integral
over $2|\tau|<|z-x_0|<\delta$ is at most $C_nL\delta$, while on
$|z-x_0|\ge\delta$ the kernel derivatives converge uniformly in $e$.
Letting first $\tau\to0$ and then $\delta\to0$ gives
\[
 \sup_{|e|=1}|Q_\tau(e)|\longrightarrow0.
\]
Hence $V$ is twice Fr\'echet differentiable at $x_0$, with
\[
 D^2V(x_0)=\int_{\R^n}D^2\Gamma_n(x_0-z)G(z)\dd z,
\]
and \eqref{eq:newton-cancel} follows.

\medskip
\noindent\emph{Step 2: the Hessian on the contact set.}
We first control the tangential component.  For almost every $y\in\cC\cap B_{2R}$,
\begin{equation}\label{eq:contact-tangent}
 |P_yD^2u(y)|\le C,
\end{equation}
where $C$ depends only on $n$, $R$, local bounds for $u$, $Du$, $N$, $DN$, and $\eta$.
Let $f:=N\eta$, so that $-\Delta u=f$.  Choose $\chi\in C_c^\infty(B_{4R})$ with $\chi\equiv1$ on $B_{3R}$.  Let $\Gamma_n$ be the fundamental solution of $-\Delta$ and define
\[
 V_y(x):=\int_{B_{4R}}\Gamma_n(x-z)\chi(z)P_yf(z)\dd z.
\]
Because $\eta$ is supported on $\cC$ and $P_yN_y=0$, the source has the pointwise cancellation
\begin{equation}\label{eq:contact-source-cancel}
 |P_yf(z)|
 =|P_y(N(z)-N_y)\eta(z)|
 \le \|DN\|_\infty\|\eta\|_\infty|z-y|
\end{equation}
for almost every $z$.  Thus the cancellation estimate \eqref{eq:newton-cancel}, applied to $G=\chi P_yf$, gives
\[
 |D^2V_y(y)|\le C.
\]
Local integrability of $\Gamma_n$, together with the uniform bound and
fixed support of $\chi P_yf$, also gives
\[
 \sup_{y\in B_{2R}}\|V_y\|_{L^\infty(B_{3R})}\le C.
\]
In $B_{3R}$ the function $P_yu-V_y$ is harmonic.  The interior harmonic estimate \cite[Theorem~2.10]{GilbargTrudinger}, together with the local $L^\infty$ bounds for $u$ and $V_y$, gives a uniform bound for its Hessian on $B_{2R}$.  Since $u\in W^{2,p}_{\rm loc}$, $D^2u(y)$ has a Lebesgue value for almost every $y$, and \eqref{eq:contact-tangent} follows.

We next control the normal component.  For almost every $y\in\cC\cap B_{2R}$,
\begin{equation}\label{eq:contact-normal}
 |N_y\cdot D^2u(y)|\le C.
\end{equation}
By \cref{prop:reaction}, $t\in W^{2,p}_{\rm loc}$ for every finite $p$.
For $w\in W^{1,p}_{\rm loc}$, the Sobolev truncation formulas
\cite[Corollary~2.1.8]{Ziemer} are
\[
 D\max\{w,0\}=\mathbf1_{\{w>0\}}Dw,
 \qquad
 D\min\{w,0\}=\mathbf1_{\{w<0\}}Dw.
\]
Since $w=\max\{w,0\}+\min\{w,0\}$, they imply
$Dw=0$ almost everywhere on $\{w=0\}$.  Applying this first to $w=t$
gives $Dt=0$ almost everywhere on $\cC=\{t=0\}$.  Each
$D_it\in W^{1,p}_{\rm loc}$ therefore vanishes almost everywhere on
$\cC$; applying the same formula to $w=D_it$ gives
$D(D_it)=0$ almost everywhere on $\cC$.  Thus
\begin{equation}\label{eq:D2t-zero}
 D^2t=0\quad\text{a.e. on }\cC.
\end{equation}
Taking the scalar product of \eqref{eq:Du-decomp} with $N$ and using
$N\cdot Dq=N\cdot DN=0$ and $|N|=1$, we obtain
\begin{equation}\label{eq:Dt-identity}
 D_it=N\cdot D_i u
\end{equation}
for each $i=1,\ldots,n$.  Since $N\in W^{1,\infty}$ and
$u\in W^{2,p}$, weak differentiation of \eqref{eq:Dt-identity} is
legitimate and yields
\begin{equation}\label{eq:D2t-source}
 D_{ji}t=D_jN\cdot D_i u+N\cdot D_{ji}u.
\end{equation}
Combining \eqref{eq:D2t-zero} and \eqref{eq:D2t-source},
\[
 |N\cdot D^2u|
 \le |DN|\,|Du|
 \quad\text{a.e. on }\cC,
\]
which proves \eqref{eq:contact-normal}.

\medskip
\noindent\emph{Step 3: the Hessian in the free phase.}
We prove that
\begin{equation}\label{eq:free-hessian}
 |D^2u(x)|\le C
\end{equation}
for every free point $x\in B_R\cap\{t>0\}$.  Indeed, the free set is
open and $u$ is smooth there because it is harmonic by
\eqref{eq:eta-support} and \eqref{eq:vector-equation}.  If the contact
set is empty in $B_{2R}$, the conclusion follows from a fixed-scale
harmonic estimate.  More generally, the same estimate applies whenever
\[
 d:=\dist(x,\cC)\ge d_0
\]
for a fixed sufficiently small $d_0<R$.  We therefore assume
$0<d<d_0$.  Since $\cC=\{t=0\}$ is relatively closed in $B$ and every
minimizing sequence for the distance from $x$ lies in
$\overline{B_{2R}}$ when $d_0\ll R$, compactness gives
$y\in\cC\cap\overline{B_{2R}}$ with $|x-y|=d$.  After decreasing
$d_0$ once more, $y\in B_{2R}$ and
$B_{2d}(y)\Subset B_{3R}$.

We first control the component tangential to $\Gamma$ at $u(y)$.  On the support of $\eta$ one has $|z-x|\ge d$, and hence
\begin{align}
 |P_yf(z)|
 &\le C|z-y|
 \le C(|z-x|+d)
 \le 2C|z-x|.
\label{eq:free-source-cancel}
\end{align}
Using the same cutoff Newton potential as in the tangential contact-set argument, now evaluated at $x$, the singular Hessian integral is bounded by
\[
 C\int_d^{4R}r^{-n}\,r\,r^{n-1}\dd r\le C.
\]
The harmonic remainder is controlled at a fixed scale.  Therefore
\begin{equation}\label{eq:free-tangent}
 |P_yD^2u(x)|\le C.
\end{equation}

For the normal component define
\[
 H(z):=N_y\cdot(u(z)-u(y)).
\]
By \eqref{eq:vector-equation}, \eqref{eq:normal-close-scale}, and $\eta\ge0$,
\begin{equation}\label{eq:H-superharmonic}
 -\Delta H=N_y\cdot N(z)\eta(z)\ge0
 \quad\text{in }B_{2d}(y),
\end{equation}
and $H$ is harmonic in $B_d(x)$.  The one-sided quadratic support \eqref{eq:quadratic-support} gives a lower bound that uses only $C^{1,1}$ geometry.  Indeed, with $q_y=q(y)=u(y)$,
\begin{align*}
 H(z)
 &=N_y\cdot(q(z)-q_y)+t(z)N_y\cdot N(z)\\
 &\ge -C|q(z)-q_y|^2
 \ge -C|u(z)-u(y)|^2
 \ge -C|z-y|^2,
\end{align*}
where the last inequality uses the local Lipschitz bound for $u$.
The inequality \eqref{eq:H-superharmonic} holds distributionally; we
use the continuous superharmonic representative of $H$, or equivalently
obtain the following supermean inequality by mollification.  Since
$H(y)=0$, it implies
\begin{equation}\label{eq:H-L1}
 \int_{B_{2d}(y)}H_+
 \le\int_{B_{2d}(y)}H_-
 \le Cd^{n+2}.
\end{equation}
For every $z\in B_{d/2}(x)$,
\[
 |z-y|\le |z-x|+|x-y|<\frac32d<2d,
 \qquad
 \dist(z,\cC)\ge d-|z-x|>\frac d2.
\]
Thus
\[
 B_{d/2}(x)\subset B_{2d}(y)\cap\{t>0\}.
\]
By \eqref{eq:eta-support} and \eqref{eq:vector-equation}, $u$, and hence
$H$, is harmonic on this ball.  Consequently $|H|$ is subharmonic there,
and the mean-value inequality gives
\[
 \sup_{B_{d/4}(x)}|H|
 \le Cd^{-n}\int_{B_{d/2}(x)}|H|.
\]
Applying the harmonic derivative estimate
\cite[Theorem~2.10]{GilbargTrudinger} on $B_{d/4}(x)$ and using
\eqref{eq:H-L1} therefore yields
\begin{equation}\label{eq:H-hessian}
 |D^2H(x)|
 \le Cd^{-n-2}\int_{B_{d/2}(x)}|H|
 \le C.
\end{equation}
Since $D^2H(x)=N_y\cdot D^2u(x)$, combining \eqref{eq:free-tangent} and \eqref{eq:H-hessian} proves \eqref{eq:free-hessian}.  The contact-set bounds \eqref{eq:contact-tangent}--\eqref{eq:contact-normal} and the free-phase bound \eqref{eq:free-hessian} together prove the proposition.
\end{proof}

\begin{proof}[Proof of \cref{thm:main}]
Fix $x_0\in\Reg_0(u)$.  By \cref{prop:lipschitz}, $u$ is locally Lipschitz near $x_0$.  If its image is separated from $\Gamma$, it is harmonic and the conclusion is classical.  Otherwise, shrink the source ball so that the tubular-coordinate analysis applies.  Proposition~\ref{prop:reaction} gives $u\in W^{2,p}_{\rm loc}$ for every finite $p$ and identifies its bounded normal reaction.  If $n=1$, the equation $-u''=N\eta\in L^\infty$ already yields the desired conclusion.  If $n\ge2$, Proposition~\ref{prop:hessian} gives
\[
 D^2u\in L^\infty_{\rm loc}
\]
near $x_0$.  Thus $u\in W^{2,\infty}_{\rm loc}$ and has a $C^{1,1}$ representative there.
If a representative $\widetilde u$ is continuous at $x_0\in\Omega$, then
\[
 \osc_{B_r(x_0)}u
 \le2\operatorname*{ess\,sup}_{B_r(x_0)}
       |\widetilde u-\widetilde u(x_0)|
 \longrightarrow0.
\]
Hence $x_0\in\Reg_0(u)$.  Applying the result just proved at $x_0$ yields
an $r>0$ such that $u\in W^{2,\infty}(B_r(x_0);\R^m)$, which is the final
assertion of \cref{thm:main}.
\end{proof}

\section{Sharpness of the \texorpdfstring{$C^{1,1}$}{C1,1} threshold}
\label{sec:sharpness}

The counterexample is based on the following projection-and-slicing
principle.

\begin{proposition}
\label{prop:shortest-arc}
Let $K\subset\R^2$ be a compact convex body, and let distinct points
$A,B\in\partial K$ be given.  Define the intrinsic arclength distance by
\[
 d_{\partial K}(A,B)
 :=\inf\left\{\operatorname{Length}(\sigma):
 \begin{array}{l}
  \sigma\in\mathrm{AC}([0,1];\partial K),\\
  \sigma(0)=A,\ \sigma(1)=B
 \end{array}\right\}.
\]
Suppose that $\gamma:[-T,T]\to\partial K$, $T>0$, is a constant-speed
parametrization satisfying
\[
 \gamma(-T)=A,
 \qquad
 \gamma(T)=B,
 \qquad
 \operatorname{Length}(\gamma)=d_{\partial K}(A,B)=:\ell.
\]
If $n=1$, set $\Omega:=(-T,T)$ and $u(t):=\gamma(t)$.  If $n\ge2$,
let $D\subset\R^{n-1}$ be a bounded Lipschitz domain and set
\[
 \Omega:=(-T,T)\times D,
 \qquad
 u(t,z):=\gamma(t).
\]
Then
\[
 u\in W^{1,\infty}
 \bigl(\Omega;\R^2\setminus\operatorname{int}K\bigr),
\]
and every
\[
 v\in W^{1,2}
 \bigl(\Omega;\R^2\setminus\operatorname{int}K\bigr),
 \qquad
 v-u\in W^{1,2}_0(\Omega;\R^2),
\]
satisfies
\[
 \int_\Omega|Dv|^2\dd x\ge\int_\Omega|Du|^2\dd x.
\]
More precisely,
\[
 \int_\Omega|Du|^2\dd x
 =\begin{cases}
   \displaystyle\frac{\ell^2}{2T},&n=1,\\[6pt]
   \displaystyle |D|\frac{\ell^2}{2T},&n\ge2.
  \end{cases}
\]
\end{proposition}

\begin{proof}
Let $v$ be an admissible competitor.  If $n\ge2$, Sobolev slicing
\cite[Theorem~2.1.4]{Ziemer} and compatibility of Sobolev traces on a
Lipschitz cylinder \cite{Leoni2017} imply that, for almost every $z\in D$,
\[
 v_z:=v(\,\cdot\,,z)\in W^{1,2}((-T,T);\R^2)
 \subset AC([-T,T];\R^2)
\]
and
\[
 (v_z-u_z)(-T)=(v_z-u_z)(T)=0,
 \qquad
 v_z(-T)=A,
 \qquad
 v_z(T)=B.
\]
Since $\R^2\setminus\operatorname{int}K$ is closed and $v_z$ belongs to
it almost everywhere, continuity gives
\[
 v_z(t)\in\R^2\setminus\operatorname{int}K
 \qquad(-T\le t\le T).
\]
For $n=1$, the same statements follow from the absolutely continuous
representative of $v$.

Let $P_K$ be the metric projection onto the closed convex set $K$.  For $x,y\in\R^2$, the variational characterization of the nearest points gives
\[
 (x-P_Kx)\cdot(P_Ky-P_Kx)\le0,
 \qquad
 (y-P_Ky)\cdot(P_Kx-P_Ky)\le0.
\]
Adding these inequalities yields
\[
 |P_Kx-P_Ky|^2
 \le(x-y)\cdot(P_Kx-P_Ky),
\]
so $P_K$ is $1$-Lipschitz.  It fixes boundary points and maps every point of
$\R^2\setminus K$ to $\partial K$; hence it maps the entire allowed set
$\R^2\setminus\operatorname{int}K$ into $\partial K$.  Consequently,
\[
 \operatorname{Length}(v(\cdot,z))
 \ge \operatorname{Length}(P_K\circ v(\cdot,z))
 \ge d_{\partial K}(A,B)=\ell.
\]
Here and in the next display, the variable $z$ is omitted when $n=1$.
Cauchy--Schwarz on $(-T,T)$ gives
\[
 \int_{-T}^T|\partial_t v(t,z)|^2\dd t
 \ge \frac{\ell^2}{2T}.
\]
For $n\ge2$, integrating over $D$ and discarding the nonnegative
transverse energy yields
\[
 \int_\Omega|Dv|^2\dd x
 \ge |D|\frac{\ell^2}{2T}
 =\int_\Omega|Du|^2\dd x,
\]
because $\gamma$ has constant speed.  The one-dimensional proof is the
same argument without slicing and gives the stated value
$\ell^2/(2T)$.
\end{proof}

\begin{proof}[Proof of \cref{thm:sharpness}]
Fix $0<\alpha<1$ and set $p:=1+\alpha$.  Consider the convex body
\begin{equation}\label{eq:Kalpha}
 K_\alpha:=\{(X,Y)\in\R^2:|X|^p+|Y|^p\le1\}.
\end{equation}
Away from the coordinate axes its boundary is smooth.  Near the top point $(0,1)$ it is the graph
\begin{equation}\label{eq:top-graph}
 Y=g(X):=(1-|X|^p)^{1/p}
 =1-\frac1p|X|^p+O(|X|^{2p}).
\end{equation}
Moreover,
\[
 g'(X)
 =-\operatorname{sgn}(X)|X|^{p-1}
   (1-|X|^p)^{1/p-1}
 =-\operatorname{sgn}(X)|X|^\alpha+o(|X|^\alpha).
\]
To verify the asserted H\"older exponent, write
\[
 g'(X)=-\psi(X)a(X),
 \qquad
 \psi(X):=\operatorname{sgn}(X)|X|^\alpha,
 \qquad
 a(X):=(1-|X|^p)^{-\alpha/p}.
\]
The function $\psi$ is $\alpha$-H\"older, while $a$ is bounded and
Lipschitz near zero.  Hence $g'$ is $\alpha$-H\"older there.  On the
other hand,
\[
 \frac{|g'(X)-g'(0)|}{|X|}\sim |X|^{\alpha-1}\longrightarrow\infty,
\]
so $g'$ is not Lipschitz at the top point.  The same analysis, after
interchanging the coordinate axes, applies at the other axis points;
away from them the boundary is smooth.  A finite cover therefore gives
\[
 \partial K_\alpha\in C^{1,\alpha}\setminus C^{1,1}.
\]

Choose $A,B\in\partial K_\alpha$ on opposite sides of $(0,1)$ and
sufficiently close to it.  The boundary arc through $(0,1)$ then realizes
$d_{\partial K_\alpha}(A,B)$: its length tends to zero as
$A,B\to(0,1)$, while the complementary length tends to the perimeter of
$\partial K_\alpha$.  Let $\gamma:[-1,1]\to\partial K_\alpha$ be its
constant-speed parametrization.  For $n=1$ take $\Omega=(-1,1)$ and
$u(t)=\gamma(t)$; for $n\ge2$ take $D=(-1,1)^{n-1}$,
$\Omega=(-1,1)\times D$, and $u(t,z)=\gamma(t)$.  By
\cref{prop:shortest-arc}, $u$ is a global energy minimizer in its
fixed-trace class.  Since a local competitor extends by $u$ to a global
competitor, $u$ is also a local minimizing constraint map.  Moreover,
\[
 u\in C^{1,\alpha}(\overline\Omega;\R^2)
 \cap W^{1,2}(\Omega;\R^2\setminus\operatorname{int}K_\alpha).
\]

It remains to prove that $u$ is not locally $C^{1,1}$.  Near $(0,1)$,
let $s=s(X)$ denote arclength along the graph.  Since
$\dd s/\dd X=\sqrt{1+|g'(X)|^2}$ is bounded above and below, the maps
$X\mapsto s(X)$ and $s\mapsto X(s)$ are locally Lipschitz.  In graph
coordinates the unit tangent is
\[
 T(X)=\frac{(1,g'(X))}{\sqrt{1+|g'(X)|^2}}.
\]
Write $\widetilde T(s)$ for the same tangent in the arclength
parametrization, so that $T(X)=\widetilde T(s(X))$.  If $\widetilde T$
were Lipschitz in $s$, then $T$ would be Lipschitz in $X$ because
$X\mapsto s(X)$ is Lipschitz.
Since $T_1\ge c>0$ near the top and $g'=T_2/T_1$, this would make $g'$
Lipschitz, a contradiction.  Hence $u$ is not $C^{1,1}$ in any
neighborhood of an interior source point mapped to $(0,1)$.  This proves
the asserted sharpness in the H\"older scale.

For completeness, the unit normal satisfies $\omega_N(r)\le Cr^\alpha$,
so
\[
 \int_0^1\frac{\omega_N(r)}{r}\dd r
 \le C\int_0^1r^{\alpha-1}\dd r
 =\frac{C}{\alpha}<\infty.
\]
Thus the same example also has $C^1$--Dini boundary regularity.
\end{proof}

\paragraph{Declaration of generative AI and AI-assisted technologies in the manuscript preparation process}

 During the preparation of this work, the authors used ChatGPT (OpenAI) to assist in converting their proof manuscripts into a preliminary LaTeX draft and to identify potential grammatical and logical errors in the manuscript. After using this tool, the authors reviewed and edited the content as needed and take full responsibility for the content of the published article.

\paragraph{Acknowledgments}

The author is supported by the National Natural Science Foundation of China [grant number: 2025YFA1017601].

\bibliographystyle{plain}
\bibliography{constraint_maps_P7_1_final}

@article{EellsSampson1964,
  author  = {Eells, Jr., James and Sampson, J. H.},
  title   = {Harmonic mappings of {R}iemannian manifolds},
  journal = {American Journal of Mathematics},
  volume  = {86},
  number  = {1},
  year    = {1964},
  pages   = {109--160},
  doi     = {10.2307/2373037}
}

@book{Morrey1966,
  author    = {Morrey, Jr., Charles B.},
  title     = {Multiple Integrals in the Calculus of Variations},
  series    = {Grundlehren der mathematischen Wissenschaften},
  volume    = {130},
  publisher = {Springer-Verlag},
  address   = {Berlin},
  year      = {1966},
  doi       = {10.1007/978-3-540-69952-1}
}

@article{HildebrandtKaulWidman1977,
  author  = {Hildebrandt, Stefan and Kaul, Helmut and Widman, Kjell-Ove},
  title   = {An existence theorem for harmonic mappings of {R}iemannian manifolds},
  journal = {Acta Mathematica},
  volume  = {138},
  number  = {1--2},
  year    = {1977},
  pages   = {1--16}
}

@article{SchoenUhlenbeck1982,
  author  = {Schoen, Richard and Uhlenbeck, Karen},
  title   = {A regularity theory for harmonic maps},
  journal = {Journal of Differential Geometry},
  volume  = {17},
  number  = {2},
  year    = {1982},
  pages   = {307--335},
  doi     = {10.4310/jdg/1214436923}
}

@article{SchoenUhlenbeck1983,
  author  = {Schoen, Richard and Uhlenbeck, Karen},
  title   = {Boundary regularity and the {D}irichlet problem for harmonic maps},
  journal = {Journal of Differential Geometry},
  volume  = {18},
  number  = {2},
  year    = {1983},
  pages   = {253--268},
  doi     = {10.4310/jdg/1214437663}
}

@article{SchoenUhlenbeck1984,
  author  = {Schoen, Richard and Uhlenbeck, Karen},
  title   = {Regularity of minimizing harmonic maps into the sphere},
  journal = {Inventiones Mathematicae},
  volume  = {78},
  number  = {1},
  year    = {1984},
  pages   = {89--100},
  doi     = {10.1007/BF01388715}
}

@article{HardtKinderlehrerLin1986,
  author  = {Hardt, Robert and Kinderlehrer, David and Lin, Fang-Hua},
  title   = {Existence and partial regularity of static liquid crystal configurations},
  journal = {Communications in Mathematical Physics},
  volume  = {105},
  number  = {4},
  year    = {1986},
  pages   = {547--570},
  doi     = {10.1007/BF01238933}
}

@article{HardtLin1987,
  author  = {Hardt, Robert and Lin, Fang-Hua},
  title   = {Mappings minimizing the {$L^p$} norm of the gradient},
  journal = {Communications on Pure and Applied Mathematics},
  volume  = {40},
  number  = {5},
  year    = {1987},
  pages   = {555--588},
  doi     = {10.1002/cpa.3160400503}
}

@article{Luckhaus1988,
  author  = {Luckhaus, Stephan},
  title   = {Partial {H}{\"o}lder continuity for minima of certain energies among maps into a {R}iemannian manifold},
  journal = {Indiana University Mathematics Journal},
  volume  = {37},
  number  = {2},
  year    = {1988},
  pages   = {349--367}
}

@article{Bethuel1993,
  author  = {B{\'e}thuel, Fabrice},
  title   = {On the singular set of stationary harmonic maps},
  journal = {Manuscripta Mathematica},
  volume  = {78},
  number  = {4},
  year    = {1993},
  pages   = {417--443},
  doi     = {10.1007/BF02599324}
}

@book{Helein2002,
  author    = {H{\'e}lein, Fr{\'e}d{\'e}ric},
  title     = {Harmonic Maps, Conservation Laws and Moving Frames},
  edition   = {2},
  series    = {Cambridge Tracts in Mathematics},
  volume    = {150},
  publisher = {Cambridge University Press},
  address   = {Cambridge},
  year      = {2002},
  doi       = {10.1017/CBO9780511543036}
}

@book{Simon1996,
  author    = {Simon, Leon},
  title     = {Theorems on Regularity and Singularity of Energy Minimizing Maps},
  series    = {Lectures in Mathematics ETH Z{\"u}rich},
  publisher = {Birkh{\"a}user},
  address   = {Basel},
  year      = {1996},
  doi       = {10.1007/978-3-0348-9193-6}
}

@article{LionsStampacchia1967,
  author  = {Lions, J.-L. and Stampacchia, Guido},
  title   = {Variational inequalities},
  journal = {Communications on Pure and Applied Mathematics},
  volume  = {20},
  number  = {3},
  year    = {1967},
  pages   = {493--519},
  doi     = {10.1002/cpa.3160200302}
}

@article{LewyStampacchia1969,
  author  = {Lewy, Hans and Stampacchia, Guido},
  title   = {On the regularity of the solution of a variational inequality},
  journal = {Communications on Pure and Applied Mathematics},
  volume  = {22},
  number  = {2},
  year    = {1969},
  pages   = {153--188},
  doi     = {10.1002/cpa.3160220203}
}

@book{KinderlehrerStampacchia1980,
  author    = {Kinderlehrer, David and Stampacchia, Guido},
  title     = {An Introduction to Variational Inequalities and Their Applications},
  series    = {Classics in Applied Mathematics},
  volume    = {31},
  publisher = {Society for Industrial and Applied Mathematics},
  address   = {Philadelphia},
  year      = {2000},
  note      = {Reprint of the 1980 original},
  doi       = {10.1137/1.9780898719451}
}

@article{DuzaarFuchs1986,
  author  = {Duzaar, Frank and Fuchs, Martin},
  title   = {Optimal regularity theorems for variational problems with obstacles},
  journal = {Manuscripta Mathematica},
  volume  = {56},
  number  = {2},
  year    = {1986},
  pages   = {209--234},
  doi     = {10.1007/BF01172157}
}

@article{Duzaar1987,
  author  = {Duzaar, Frank},
  title   = {Variational inequalities and harmonic mappings},
  journal = {Journal f{\"u}r die reine und angewandte Mathematik},
  volume  = {374},
  year    = {1987},
  pages   = {39--60},
  doi     = {10.1515/crll.1987.374.39}
}

@article{FKSobstacle,
  author  = {Figalli, Alessio and Kim, Sunghan and Shahgholian, Henrik},
  title   = {Constraint maps with free boundaries: the obstacle case},
  journal = {Archive for Rational Mechanics and Analysis},
  volume  = {248},
  number  = {5},
  year    = {2024},
  pages   = {Paper No. 79},
  doi     = {10.1007/s00205-024-02032-5}
}

@article{FGKSreview,
  author  = {Figalli, Alessio and Guerra, Andr{\'e} and Kim, Sunghan and Shahgholian, Henrik},
  title   = {Constraint maps: insights and related themes},
  journal = {La Matematica},
  volume  = {5},
  year    = {2026},
  pages   = {Article 26},
  doi     = {10.1007/s44007-026-00209-w}
}

@article{Federer1959,
  author  = {Federer, Herbert},
  title   = {Curvature measures},
  journal = {Transactions of the American Mathematical Society},
  volume  = {93},
  number  = {3},
  year    = {1959},
  pages   = {418--491},
  doi     = {10.1090/S0002-9947-1959-0110078-1}
}

@article{LeobacherSteinicke2021,
  author  = {Leobacher, Gunther and Steinicke, Alexander},
  title   = {Existence, uniqueness and regularity of the projection onto differentiable manifolds},
  journal = {Annals of Global Analysis and Geometry},
  volume  = {60},
  year    = {2021},
  pages   = {559--587},
  doi     = {10.1007/s10455-021-09788-z}
}

@book{GilbargTrudinger,
  author    = {Gilbarg, David and Trudinger, Neil S.},
  title     = {Elliptic Partial Differential Equations of Second Order},
  series    = {Classics in Mathematics},
  publisher = {Springer},
  address   = {Berlin},
  year      = {2001}
}

@book{Leoni2017,
  author    = {Leoni, Giovanni},
  title     = {A First Course in Sobolev Spaces},
  edition   = {2},
  series    = {Graduate Studies in Mathematics},
  volume    = {181},
  publisher = {American Mathematical Society},
  address   = {Providence, RI},
  year      = {2017}
}

@book{Rudin1976,
  author    = {Rudin, Walter},
  title     = {Principles of Mathematical Analysis},
  edition   = {3},
  publisher = {McGraw--Hill},
  address   = {New York},
  year      = {1976}
}

@book{Ziemer,
  author    = {Ziemer, William P.},
  title     = {Weakly Differentiable Functions},
  series    = {Graduate Texts in Mathematics},
  volume    = {120},
  publisher = {Springer},
  address   = {New York},
  year      = {1989}
}

@book{Giaquinta1983,
  author    = {Giaquinta, Mariano},
  title     = {Multiple Integrals in the Calculus of Variations and Nonlinear Elliptic Systems},
  series    = {Annals of Mathematics Studies},
  volume    = {105},
  publisher = {Princeton University Press},
  address   = {Princeton, NJ},
  year      = {1983},
  doi       = {10.1515/9781400881628}
}

@article{Evans1991,
  author  = {Evans, Lawrence C.},
  title   = {Partial regularity for stationary harmonic maps into spheres},
  journal = {Archive for Rational Mechanics and Analysis},
  volume  = {116},
  number  = {2},
  year    = {1991},
  pages   = {101--113},
  doi     = {10.1007/BF00375587}
}

@article{Lin1999,
  author  = {Lin, Fang-Hua},
  title   = {Gradient estimates and blow-up analysis for stationary harmonic maps},
  journal = {Annals of Mathematics},
  volume  = {149},
  number  = {3},
  year    = {1999},
  pages   = {785--829},
  doi     = {10.2307/121073}
}

@article{Riviere2007,
  author  = {Rivi{\`e}re, Tristan},
  title   = {Conservation laws for conformally invariant variational problems},
  journal = {Inventiones Mathematicae},
  volume  = {168},
  number  = {1},
  year    = {2007},
  pages   = {1--22},
  doi     = {10.1007/s00222-006-0023-0}
}

@book{LinWang2008,
  author    = {Lin, Fang-Hua and Wang, Changyou},
  title     = {The Analysis of Harmonic Maps and Their Heat Flows},
  publisher = {World Scientific},
  address   = {Hackensack, NJ},
  year      = {2008},
  doi       = {10.1142/9789812779533}
}

@article{Caffarelli1998,
  author  = {Caffarelli, Luis A.},
  title   = {The obstacle problem revisited},
  journal = {Journal of Fourier Analysis and Applications},
  volume  = {4},
  number  = {4--5},
  year    = {1998},
  pages   = {383--402},
  doi     = {10.1007/BF02498216}
}

@book{CaffarelliSalsa2005,
  author    = {Caffarelli, Luis A. and Salsa, Sandro},
  title     = {A Geometric Approach to Free Boundary Problems},
  series    = {Graduate Studies in Mathematics},
  volume    = {68},
  publisher = {American Mathematical Society},
  address   = {Providence, RI},
  year      = {2005}
}

@book{PetrosyanShahgholianUraltseva2012,
  author    = {Petrosyan, Arshak and Shahgholian, Henrik and Uraltseva, Nina},
  title     = {Regularity of Free Boundaries in Obstacle-Type Problems},
  series    = {Graduate Studies in Mathematics},
  volume    = {136},
  publisher = {American Mathematical Society},
  address   = {Providence, RI},
  year      = {2012},
  doi       = {10.1090/gsm/136}
}

@article{BethuelBrezis1991,
  author  = {B{\'e}thuel, Fabrice and Brezis, Ha{\"i}m},
  title   = {Regularity of minimizers of relaxed problems for harmonic maps},
  journal = {Journal of Functional Analysis},
  volume  = {101},
  number  = {1},
  year    = {1991},
  pages   = {145--161},
  doi     = {10.1016/0022-1236(91)90152-U}
}

\bigskip
\noindent
(Yilu Liu) School of Mathematical Sciences, University of Science and
Technology of China, Hefei, 230026, Anhui Province, P.R. China.\\
Email address: \href{mailto:liuylgeoanaly@mail.ustc.edu.cn}{liuylgeoanaly@mail.ustc.edu.cn}

\end{document}